\documentclass[12pt]{article}
\usepackage{amsmath,mathdots}
\usepackage{amssymb}
\usepackage{latexsym}
\usepackage{amsthm}
\usepackage{amsmath,amssymb,amsfonts,amsthm}
\usepackage{mdframed}
\usepackage{graphicx}
\usepackage{stmaryrd}
\usepackage{gensymb}
\usepackage{bm}
\usepackage{tikz}  
\usepackage{longtable}
\usepackage{stmaryrd}
\usepackage{multirow}
\usepackage{float}
\usepackage{array}
\usepackage{graphicx}
\usepackage{tikz}
\usepackage{enumerate}
\usepackage{enumitem}
\usetikzlibrary{matrix,arrows}
\usetikzlibrary{positioning}
\usetikzlibrary{fit}
\usetikzlibrary{patterns}

\usepackage[colorlinks,
linkcolor=blue,
anchorcolor=blue,
citecolor=blue
]{hyperref}
\usepackage[margin=2.9cm]{geometry}
\usepackage[normalem]{ulem}

\newtheorem{thm}{Theorem}[section]
\newtheorem{defn}[thm]{Definition}
\newtheorem{lem}[thm]{Lemma}

\newtheorem{coro}[thm]{Corollary}
\newtheorem{conj}{Conjecture}
\newtheorem{example}[thm]{Example}

\newtheorem{remark}[thm]{Remark}

\usepackage{graphicx}
\usepackage{tikz}
\usetikzlibrary{matrix,arrows}
\usetikzlibrary{positioning}
\usetikzlibrary{fit}
\usetikzlibrary{patterns}

\def\fS{\mathfrak{S}}
\def\fG{\mathfrak{G}}
\def\wilf{\mathsf{wilf}}
\def\dist{\mathsf{dist}}
\def\anc{\mathrm{anc}}
\def\mla{\mathrm{mla}}
\def\rr{\mathrm{r}}
\def\RLM{\mathrm{RLMax}}

\newcommand{\pattern}[4]{										% mesh pattern
	\raisebox{0.6ex}{
		\begin{tikzpicture}[scale=0.35, baseline=(current bounding box.center), #1]
		\foreach \x/\y in {#4}		\fill[gray!50] (\x,\y) rectangle +(1,1);
		\draw (0.01,0.01) grid (#2+0.99,#2+0.99);
		\foreach \x/\y in {#3}		\filldraw (\x,\y) circle (6pt);
		\end{tikzpicture}}
}

\definecolor{red}{rgb}{1,0,0}
{}

 \allowdisplaybreaks

\begin{document}

\begin{center}
{\large On mesh patterns of short length: Equidistribution and enumeration}
\end{center}

\begin{center}
{\small
Qi Fang$^{a}$, Shishuo Fu$^{b}$, Sergey Kitaev$^{c}$, Haijun Li$^{d}$, Xinyu Su$^{e}$ and Ziyao Sun$^{f,\dagger}$
\\[6pt]

$^{a, b, d, f}$College of Mathematics and Statistics, Chongqing University,\\ Chongqing 401331, PR China

$^{b}$Center for Discrete Mathematics, Chongqing University,\\ Chongqing 401331, PR China

$^{c}$Department of Mathematics and Statistics, University of Strathclyde, \\ 26 Richmond Street, Glasgow G1 1XH, UK\\[6pt]

$^{e}$College of Mathematical Sciences \& Institute of Mathematics and Interdisciplinary Sciences, Tianjin Normal University, \\ Tianjin  300387, P. R. China\\[6pt]

$^\dagger$Corresponding author: Ziyao Sun\\[6pt]

Email:$^{a}${\tt  qifangpapers@163.com},
           $^{b}${\tt  fsshuo@cqu.edu.cn},
           $^{c}${\tt sergey.kitaev@strath.ac.uk},
           $^{d}${\tt  lihaijun@cqu.edu.cn},
           $^{e}${\tt suxinyu2220@163.com},
           $^{f}${\tt  sunziyao79@gmail.com}
           }
\end{center}

\noindent\textbf{Abstract.} The classification and enumeration of short mesh patterns have emerged as two central directions in the area. We make substantial progress on both fronts. We construct an involution and a bijection that establish distributional equivalences for two classes of length-$2$ mesh patterns, thereby resolving a conjecture from 2019 and a recent conjecture. As a consequence, the best known upper bounds for the numbers of distribution-equivalence and Wilf-equivalence classes drop to $106$ and $47$, respectively. Combined with the known lower bounds of 105 and 46, conjectured to be exact, these results leave both classifications hinging on a single distribution-equivalence question conjectured in 2019, whose resolution would at once settle the remaining Wilf-equivalence case. We further conjecture that this unresolved equidistribution also holds for involutions, a subclass of all permutations.

We also determine the distributions of three additional classes of length-$2$ mesh patterns through a detailed structural analysis. Our work combines bijective techniques with generating-function methods, yielding new insights into the structure and enumeration of short mesh patterns.



\section{Introduction}\label{intro}

In this paper, $[a,b]$ denotes the set of integers from $a$ to $b$, that is, $[a,b] := \{a, a+1, \ldots, b\}$. In particular we write $[n]:=[1,n]$. We denote the set of permutations of $[n]$ by $\fS_n$, express each permutation $\sigma\in\fS_n$ using its one-line notation $\sigma_1\ldots\sigma_n$ with $\sigma_i=\sigma(i)$ for $1\le i\le n$, and refer to $\sigma$ as an $n$-permutation. A (classical) \emph{permutation pattern} is a shorter permutation. Given two permutations $\sigma\in\fS_n$ and $\pi\in\fS_m$, we say that $\sigma$ \emph{contains} $\pi$ as a pattern if some subsequence of the entries of $\sigma$ (again, in its one-line expression) has the same relative order as all of the entries of $\pi$. Otherwise $\sigma$ is said to \emph{avoid} $\pi$ as a pattern. For instance, the permutation $\sigma=216435$ contains both $132$ and $321$ as patterns, while it avoids the pattern $231$. Besides its well‑known origin in the stack‑sorting problem of algorithm analysis~\cite{Knu1997}, the permutation pattern has proven strikingly versatile, characterizing structures across many areas of mathematics. A quite recent illustration is the following remarkable theorem in Schubert calculus, due to Chen, Fan, and Ye:
\begin{thm}[{\cite[Theorem~1.1]{CFY2026}}]
The Grothendieck polynomial $\fG_w(x)$ is zero-one, i.e., has coefficients $0$ or $\pm 1$, if and only if $w$ avoids the patterns $1432$, $1342$, $13254$, $31524$, $12534$, and $21534$.
\end{thm}

As research on patterns progressed, it quickly became apparent that extending the original notion of permutation patterns was both necessary and beneficial. The notion of {\it mesh pattern} was the outcome of one such endeavour by Br\"and\'en and Claesson~\cite{BC2011}; see Definition~\ref{def:mesh pattern}. It subsumes several classes of patterns as special cases, such as consecutive, vincular, and bivincular patterns. For undefined terms and further information, see the monograph~\cite{Kit2011} by the third author.

From our perspective, research on mesh patterns primarily focuses on two major classes of problems, and the present paper makes contributions to both. The first class of problems is concerned with the classification. For all $n,k\ge 0$, let $\fS_n(p)$ denote the set of $n$-permutations that avoid a given pattern $p$, and let $s_{n,k}(p)$ be the number of $n$-permutations with $k$ occurrences of $p$. In particular, we see that $s_{n,0}(p)=|\fS_n(p)|$. Given two patterns $p_1$ and $p_2$, we denote by $\fS_n(p_1,p_2)$ the set of $n$-permutations that avoid simultaneously both patterns $p_1$ and $p_2$. We say that $p_1$ and $p_2$ are \emph{Wilf-equivalent} if $|\fS_n(p_1)|=|\fS_n(p_2)|$ for all $n\ge 0$. They are said to be \emph{equidistributed}, if the condition $s_{n,k}(p_1)=s_{n,k}(p_2)$ holds for all $n,k\ge 0$, in which case we write $p_1\sim p_2$. If we use $p(\sigma)$ to denote the number of occurrences of the pattern $p$ in a permutation $\sigma$, then the relation $p_1\sim p_2$ can be rephrased in terms of generating functions as 
\begin{align*}
  \sum_{n\ge 0}t^n\sum_{\sigma\in\fS_n}q^{p_1(\sigma)} &= \sum_{n\ge 0}t^n\sum_{\sigma\in\fS_n}q^{p_2(\sigma)}.
\end{align*}

A classification of mesh patterns of length $2$ was initiated by Hilmarsson et al.~\cite{Hilmarsson2015Wilf}, who showed that the number of Wilf-equivalence classes for length-$2$ mesh patterns, denoted here by $\wilf$, is at most $56$, and conjectured that $\wilf=46$. This upper bound was improved to $\wilf \le 49$ in a recent work by Su, Kitaev, and Zhang~\cite{SKZ2026}. The same work~\cite{SKZ2026} also provided a nearly complete classification of equidistribution equivalence classes for mesh patterns of length $2$. The authors showed that the number of such classes, denoted here by $\dist$, is at most $108$, and conjectured that $\dist=105$. Interestingly, the patterns involved in the conjectured equivalences are the same as those appearing in the conjectured Wilf-equivalences, so that proving any of the equidistribution conjectures automatically implies the corresponding Wilf-equivalence conjectures.  

In the present work, we further reduce these upper bounds to $\wilf \le 47$ and $\dist \le 106$, respectively, by resolving a 2019 conjecture from~\cite{KZ2019} and a recent conjecture from~\cite{SKZ2026}. Unless otherwise noted, the class numbers used throughout the paper are those catalogued in~\cite{SKZ2026}.

\begin{thm}\label{thm:classification}
For mesh patterns of length $2$, the number $\dist$ of equidistribution equivalence classes satisfies $105 \le \dist \le 106$. 
\end{thm}

\begin{remark}\label{class-69-remark-equivalence} The only remaining conjectural equidistributed class is Class $69$ comprising of the patterns in the set $\{\pattern{scale=0.5}{2}{1/1,2/2}{1/2,1/1,2/1,0/0},\pattern{scale=0.5}{2}{1/1,2/2}{2/2,0/1,1/1,1/0},\pattern{scale=0.5}{2}{1/1,2/2}{0/2,1/1,2/1,1/0},\pattern{scale=0.5}{2}{1/1,2/2}{1/2,0/1,1/1,2/0} \}$, where the first two patterns, as well as the last two patterns, are trivially equivalent.
\end{remark}

Since conjectured equidistributions imply the corresponding Wilf-equivalences, we immediately obtain the following improvement on the value of $\wilf$.

\begin{coro}
For mesh patterns of length $2$, the number $\wilf$ of Wilf-equivalence classes satisfies $46 \le \wilf \le 47$.
\end{coro}

\begin{remark}\label{class-69-remark-Wilf-equivalence}
The only remaining conjectural Wilf-equivalence class is Class $69$ given in Remark~\ref{class-69-remark-equivalence}.
\end{remark}

Building on the work of \cite{SKZ2026}, the proof of Theorem~\ref{thm:classification} reduces to two equidistribution results, stated and proved in Section~\ref{sec:two_equidistributions}.

The second class of problems frequently raised for mesh patterns (or rather for any kind of patterns in general) is enumerative in nature. Namely, for a fixed (mesh) pattern $p$, one would like to compute the bivariate generating function
$$F^p(t,q) := \sum_{n,k\geq 0}s_{n,k}(p) q^kt^n.$$

For simpler notation, we set $A^p(t):=F^p(t,0)$, $F(t):=F^p(t,1)=\sum_{n\ge 0}n! t^n$, and 
\begin{align*}
F^p_{\geq 1}(t,q) &:= F^p(t,q)-A^p(t)=\sum_{n\geq 0,\, k\ge 1}s_{n,k}(p)q^k t^n.
\end{align*}
% Moreover, we use $A(x)$ (resp., $F(x,q)$) to denote the ordinary generating function for permutations avoiding the pattern in question (resp., the generating function for the distribution of occurrences of the pattern). Hence, if $p(\pi)$ denotes the number of occurrences of a pattern $p$ in a permutation $\pi$, then
%In what follows, we abbreviate ``generating function'' as g.f., 

Note that the superscript ``$p$'' is usually dropped when the pattern is understood, and we use $[t^n]F(t,q)$ to denote the coefficient of $t^n$ in $F(t,q)$. The following three theorems provide three of the previously unknown distributions listed in \cite{SKZ2026}. To be precise, we obtain the distributions for any pattern $p$ in Classes 54, 60, and 75, respectively.

\begin{thm}\label{thm:class54}
For any pattern $p$ in Class $54$, namely, for
$$
p\in \{\pattern{scale=0.6}{2}{1/1,2/2}{0/2,2/2,1/1,2/1,0/0}, 
            \pattern{scale=0.6}{2}{1/1,2/2}{0/2,2/2,1/1,0/0,1/0}, 
            \pattern{scale=0.6}{2}{1/1,2/2}{1/2,2/2,1/1,0/0,2/0}, 
            \pattern{scale=0.6}{2}{1/1,2/2}{2/2,0/1,1/1,0/0,2/0},
            \pattern{scale=0.6}{2}{1/1,2/2}{0/2,2/2,0/1,1/1,2/0}, 
            \pattern{scale=0.6}{2}{1/1,2/2}{0/2,1/2,1/1,0/0,2/0}, 
            \pattern{scale=0.6}{2}{1/1,2/2}{0/2,1/1,2/1,0/0,2/0}, 
            \pattern{scale=0.6}{2}{1/1,2/2}{0/2,2/2,1/1,1/0,2/0}\},$$
we have $A(t)=F(t)-F_{\geq 1}(t, 1)$ and
\begin{align}\label{class54-distr}
F_{\geq 1}(t,q) &= \sum_{n\geq 0}t^nq\left[\sum_{\substack{i, j, k, m\geq 0\\ i+j+k+m=n-2}} \binom{j+k}{k}(i+j)! T_m S_k\right.\nonumber\\
&\left.+\sum_{\substack{i, j_1, j_2, k_1, k_2, m\geq 0\\ i+j_1+j_2+k_1+k_2+m=n-4}}\binom{k_2+j_2+1}{k_2}\binom{k_1+j_1}{k_1}k_1!(i+j_1+j_2+1)! T_m S_{k_2}\right],
\end{align}
where
\begin{align*}
T_m:=T_m(q)=[t^m]\frac{F(t)}{1+t(1-q)F(t)}\text{ and }S_k:=[t^k]\frac{F(t)}{1+tF(t)}.
\end{align*}
\end{thm}

\begin{thm}\label{thm:class60}
For any pattern $p$ in Class $60$, namely, for
$$
p\in\{\pattern{scale=0.6}{2}{1/1,2/2}{0/2,1/2,2/2,1/1,0/0},\pattern{scale=0.6}{2}{1/1,2/2}{0/2,2/2,0/1,1/1,0/0},\pattern{scale=0.6}{2}{1/1,2/2}{2/2,1/1,2/1,0/0,2/0},\pattern{scale=0.6}{2}{1/1,2/2}{2/2,1/1,0/0,1/0,2/0},\pattern{scale=0.6}{2}{1/1,2/2}{0/2,1/2,0/1,2/1,0/0},\pattern{scale=0.6}{2}{1/1,2/2}{0/2,1/2,2/2,0/1,1/0},\pattern{scale=0.6}{2}{1/1,2/2}{2/2,0/1,2/1,1/0,2/0},\pattern{scale=0.6}{2}{1/1,2/2}{1/2,2/1,0/0,1/0,2/0}\},$$
we have $A(t)=F(t,0)$ and for $n\ge 1$,
\begin{equation}\label{class60-distr}
[t^n] F(t,q)=\sum_{m=0}^{n-1}\frac{(n-1)!}{m!}T_m,
\end{equation}
where $T_m$ is defined in Theorem~\ref{thm:class54}.
\end{thm}

\begin{thm}\label{thm:class75}
For any pattern $p$ in Class $75$, namely, for
$$
p\in\{\pattern{scale=0.6}{2}{1/1,2/2}{1/2,2/2,0/1}, 
\pattern{scale=0.6}{2}{1/1,2/2}{1/2,0/1,0/0},
\pattern{scale=0.6}{2}{1/1,2/2}{2/2,2/1,1/0}, 
\pattern{scale=0.6}{2}{1/1,2/2}{2/1,1/0,0/0},
\pattern{scale=0.6}{2}{1/1,2/2}{2/2,0/1,1/1},
\pattern{scale=0.6}{2}{1/1,2/2}{1/2,1/1,0/0},
\pattern{scale=0.6}{2}{1/1,2/2}{1/1,2/1,0/0},
\pattern{scale=0.6}{2}{1/1,2/2}{2/2,1/1,1/0},
\pattern{scale=0.6}{2}{1/1,2/2}{1/2,2/2,0/0},
\pattern{scale=0.6}{2}{1/1,2/2}{2/2,0/1,0/0},
\pattern{scale=0.6}{2}{1/1,2/2}{2/2,2/1,0/0},
\pattern{scale=0.6}{2}{1/1,2/2}{2/2,0/0,1/0}\},
$$
$s_{n,0}(p)=1$ for all $n\geq 0$, and $s_{0,k}(p)=0$ for all $k\geq 1$. Moreover, for $n>0$ and $k\geq 0$, 
\begin{align}\label{thm:class75-rec-1}
s_{n,k}(p)
&=
n s_{n-1,k}(p)-\sum_{h=1}^{n-1} h \cdot s_{n-1,k,h}(p)
+\sum_{h=1}^{n-1} h \cdot s_{n-1,k-1,h}(p).
\end{align}
Here, $s_{n,k,h}(p)$ denotes the number of $n$-permutations with exactly $k$ occurrences of $p$ and exactly $h$ admissible positions (see Definition~\ref{def:admissible pos}). These numbers satisfy
\begin{align}
s_{n,k,1}(p) &= s_{n,k}(p) - \sum_{h=2}^{n} s_{n,k,h}(p), \label{thm:class75-rec-2} \\
s_{n,k,h}(p) &= \sum_{i=1}^{n-1}\sum_{a=0}^{k}s_{i,a,1}(p)\cdot s_{n-i,k-a,h-1}(p),  \text{ for } h\geq 2. \label{thm:class75-rec-3}
\end{align}
The initial conditions are $s_{0,0,0}(p)=1$, and $s_{0,k,h}(p)=0$ whenever $(k,h)\neq(0,0)$, while for $n\geq 1$, we have $s_{n,k,h}(p)=0$ whenever $k\not\in [0,n-1]$, or $h\not\in [1,n]$.
\end{thm}

The proofs of the above three theorems are given in Section~\ref{sec:proofs_of_the_three_distributions}. We remark that the enumeration of the avoiders for Class 75, given by the trivial sequence $1,1,1,\ldots$, was already noted in \cite[Tab.~1]{Hilmarsson2015Wilf}; see also \cite[Thm.~7.3]{SKZ2026} for more patterns that are Wilf-equivalent to Class 75. Moreover, Classes 30 and 60 were shown to be Wilf-equivalent and were included in Class 46 of \cite[Tab.~10]{Hilmarsson2015Wilf}. Our Theorem~\ref{thm:class60} thus not only provides the distribution for Class 60, but also in effect gives a formula for the avoiders for Class~30. Our contributions are summarized in Table~\ref{tab:summary} below, where ``A'' indicates that the avoidance formula is now known, while ``D'' indicates that the full distribution has been determined.

% In summary, our present work resolves two open avoidance problems from \cite{Hilmarsson2015Wilf} (Wilf-equivalence Classes 46 and 51), replaces three question marks in Table~6 of \cite{SKZ2026} (two by ``D'' and one by ``A'') and one ``A'' by ``D'', and, in particular, solves three of the 72 open distribution problems posed in \cite{SKZ2026}. 

% Moreover, as noted in the caption of Table~6 in \cite{SKZ2026}, the Shading Lemma from \cite{Hilmarsson2015Wilf} implies that Classes 30 and 60 are Wilf-equivalent. Class 30 consists of the patterns in the set $\{\pattern{scale=0.6}{2}{1/1,2/2}{0/2,1/2,2/1,0/0},\pattern{scale=0.6}{2}{1/1,2/2}{0/2,2/2,0/1,1/0},\pattern{scale=0.6}{2}{1/1,2/2}{1/2,2/1,0/0,2/0},\pattern{scale=0.6}{2}{1/1,2/2}{2/2,0/1,1/0,2/0}\}$. Consequently,  also yields the avoidance formula for Class 30. Thus, our results resolve multiple previously unknown entries in Table~6 of \cite{SKZ2026}.

\medskip

\begin{table}[h]
    \renewcommand{\arraystretch}{0.9}
    \centering
\begin{tabular}{|c|c|c|c|}
\hline
Class No.~in \cite{SKZ2026} & Previous status from \cite[Tab.~6]{SKZ2026} & New status & Ref. \\
\hline
30 & ? & A & Thm~\ref{thm:class60}\\
54 & A & D & Thm~\ref{thm:class54} \\
60 & ? & D & Thm~\ref{thm:class60}\\
75 & A & D & Thm~\ref{thm:class75} \\
\hline
\end{tabular}
\caption{A summary of our enumerative contributions.}\label{tab:summary}
\end{table}

% \medskip

% Here,  Also, Class 54 consists of the patterns in the set $\{\pattern{scale=0.6}{2}{1/1,2/2}{0/2,2/2,1/1,2/1,0/0},\pattern{scale=0.6}{2}{1/1,2/2}{0/2,2/2,1/1,0/0,1/0},\pattern{scale=0.6}{2}{1/1,2/2}{1/2,2/2,1/1,0/0,2/0},\pattern{scale=0.6}{2}{1/1,2/2}{2/2,0/1,1/1,0/0,2/0},\pattern{scale=0.6}{2}{1/1,2/2}{0/2,2/2,0/1,1/1,2/0},\pattern{scale=0.6}{2}{1/1,2/2}{0/2,1/2,1/1,0/0,2/0} ,\pattern{scale=0.6}{2}{1/1,2/2}{0/2,1/1,2/1,0/0,2/0},\pattern{scale=0.6}{2}{1/1,2/2}{0/2,2/2,1/1,1/0,2/0}\}$. 

% Additionally, our distribution formula for Class 54 from \cite{SKZ2026}, which also resolved an equidistribution conjecture, simultaneously solves the open avoidance problem for Wilf-equivalence Class 51 in \cite{Hilmarsson2015Wilf}.

%%%%%%%%%%%%%%%%%%%%%%%%%%%%%%%%%%%%%%%%%%%%%%%%%%%%%%%%%%%%%%%%%%%%%%%%%
\section{Two equidistributions}\label{sec:two_equidistributions}
We begin with a formal definition of a mesh pattern.
\begin{defn}\label{def:mesh pattern}
 A \emph{mesh pattern} is a pair $p = (\pi, R)$, where $\pi = \pi_1 \pi_2 \cdots \pi_m \in \mathfrak{S}_m$ and $R \subseteq [0,m] \times [0,m]$. It can be represented pictorially using a grid diagram decorated with $m$ heavy dots and $|R|$ shaded unit cells. The dots are placed at coordinates $(i, \pi_i)$ for $1 \le i \le m$, and a unit cell is shaded if and only if the coordinates pair of its bottom-left corner belongs to $R$.
\end{defn} 

Ignoring the shadings and focusing on the relative heights of the dots recovers the familiar notion of a permutation matrix (see, e.g., \cite[Section~1.5]{EC1}) for the underlying permutation $\pi$. For example, the mesh pattern $p = (132, \{(0,0), (2,3), (3,1), (3,2)\})$ is illustrated below:
$$
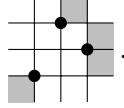
\pattern{scale=1}{3}{1/1,2/3,3/2}{0/0,2/3,3/1,3/2}\, .
$$

Here and henceforth, when we say that $(a,b)$ is an occurrence of a mesh pattern of length $2$, $a$ and $b$ refer to the {\em values} of the underlying (classical) pattern (i.e., the two ``dots'' in the pictorial representation of the pattern). Thus, for a mesh pattern with underlying $12$-pattern occurring in a given permutation $\sigma$, we have $a<b$ and $\sigma^{-1}(a)<\sigma^{-1}(b)$. The proof of Theorem~\ref{thm:classification} reduces to Theorems~\ref{thm:class54-equi} (for Class 54) and \ref{thm:class71-equi} (for Class 71), each confirming a previously conjectured equidistribution. Theorem~\ref{thm:classification} then follows, since it is already known from \cite{SKZ2026} that $105 \le \dist \le 108$. 

We begin with Class 54, which, as shown in \cite[Tab.~5]{SKZ2026}, consists of $8$ patterns divided into two subclasses (each subclass contains $4$ patterns). The $4$ patterns within each subclass are trivially equidistributed via elementary operations such as reverse-complementation and inversion. The nontrivial task is to establish distributional equivalence between two patterns coming from different subclasses. Let us consider
\[
p_1=\pattern{scale=0.6}{2}{1/1,2/2}{0/0,0/1,1/1,2/0,2/2}
\quad \text{and} \quad
p_2=\pattern{scale=0.6}{2}{1/1,2/2}{0/2,0/1,1/1,2/0,2/2},
\]
indeed two patterns from separate subclasses. We construct a value-shifting involution $\phi: \mathfrak{S}_n \to \mathfrak{S}_n$, which shows that $p_1$ and $p_2$ are in fact \emph{jointly equidistributed} in the following sense.

\begin{thm}\label{thm:class54-joint equi}
With $p_1$ and $p_2$ as given above and every $n\ge 1$, the pair of patterns $(p_1,p_2)$ is equidistributed with the pair $(p_2,p_1)$ over $\fS_n$. In particular, $p_1\sim p_2$.
\end{thm}

From Theorem~\ref{thm:class54-joint equi}, we deduce the following result which confirms Conjecture 2 of \cite{SKZ2026}.
\begin{thm}\label{thm:class54-equi}
The eight patterns in Class $54$ are all equidistributed.
\end{thm}

To facilitate our construction of $\phi$, we first establish the structural rigidity pertaining to the occurrences of $p_1$ and $p_2$ patterns through the following two lemmas.
\begin{lem}\label{lem:unique-a}
  If a permutation $\sigma$ contains at least one occurrence of $p_1$ (resp., $p_2$), then the element ``$a$'' acting as the first element of the occurrence is uniquely determined.
\end{lem}
\begin{proof}
We only prove the claim for the pattern $p_1$, since the proof for $p_2$ is based on similar ideas and is left to the interested reader. If $\sigma$ contains one and only one occurrence of $p_1$ then the claim trivially holds true. So we can assume that $(a, b)$ and $(a', b')$ are two distinct occurrences of $p_1$ in $\sigma$, and aim to show that $a=a'$.

Firstly, if $b=b'$, then we must also have $a=a'$ due to the shadings of $p_1$ lying horizontally between the two dots, leading to a contradition with $(a, b)$ and $(a', b')$ being distinct. Therefore we must have $b\neq b'$. Now suppose on the contrary that $a\neq a'$ and $\sigma^{-1}(a)<\sigma^{-1}(a')$ (if $\sigma^{-1}(a)>\sigma^{-1}(a')$ then we simply swap their roles). By the restriction of $p_1$ on $a'$, every element to the left of $a'$ must be larger than $b'$ in value. Since we have assumed that $a$ is to the left of $a'$, we have $b' < a < b$. By the restriction of $p_1$ on $b'$, every element larger than $b'$ must be to the left of $b'$. So $b>a>b'$ effectively implies that $b'$ lies positionally to the right of $b$ but its value is outside of the interval $[a,b]$. Hence $(a,b)$ cannot be a valid occurrence of $p_1$, yielding a contradiction. In conclusion, we should have $a=a'$ and this element $a$ is indeed uniquely determined.
\end{proof}

\begin{lem}\label{lem:same-a}
  If $\sigma$ contains at least one occurrence of $p_1$, say $(a_1, b_1)$, and at least one occurrence of $p_2$, say $(a_2, b_2)$, then we have that
  \begin{enumerate}[label=(\arabic*)]
    \item $a_1 = a_2$;
    \item $\sigma^{-1}(a_1)=1$, and the pair $(a_1,b)$ constitutes an occurrence of $p_1$ in $\sigma$ if and only if $(a_1,b)$ constitutes an occurrence of $p_2$ in $\sigma$. In particular, we have $p_1(\sigma)=p_2(\sigma)$.
  \end{enumerate}
\end{lem}

\begin{proof}
For part (1), let us first consider the case that $\sigma^{-1}(a_2) < \sigma^{-1}(a_1)$. By the same argument as in the proof of Lemma~\ref{lem:unique-a}, we can conclude that $b_1$ lies positionally to the right of $b_2$ but outside of the interval $[a_2,b_2]$ in value, which prevents $(a_2,b_2)$ from being a valid $p_2$ pattern, a contradiction. For the other case with $\sigma^{-1}(a_1) < \sigma^{-1}(a_2)$, by the restriction of $p_2$ on $a_2$ we see that $a_1<a_2$. Then due to the shaded cell in $p_1$ at $(1,1)$ and the shaded cell in $p_2$ at $(2,0)$, we must have $b_1<a_2<b_2$ and $b_1$ is to the left of $b_2$. But this makes $(a_1,b_1)$ an invalid occurrence of $p_1$, again a contradiction. This finishes the proof of (1).

To see (2), note that the restriction of $p_1$ requires that all elements to the left of $a_1$ must be larger than $a_1$, whereas the restriction of $p_2$ forces all elements to the left of $a_2$ ($=a_1$ by (1)) to be smaller than $a_2$. These two requirements are mutually exclusive, leaving us with the only possibility that there are no elements to the left of $a_1=a_2$, whence both requirements are vacuously satisfied. Thus we have $\sigma^{-1}(a_1)=1$ as claimed. Furthermore, in view of the fact that $p_1$ and $p_2$ have the identical shadings for their 2nd and 3rd columns, we obtain the ``if and only if'' part of the claim.
\end{proof}
In view of the last two lemmas, whenever $\sigma \notin \fS_n(p_1,p_2)$, all occurrences of the patterns $p_1$ or $p_2$ must begin with the uniquely determined element $a$. We refer to $a$ as the \emph{anchor} of $\sigma$ and denote it by $\anc(\sigma)$.

\begin{proof}[Proof of Theorem~\ref{thm:class54-joint equi}]
It suffices to define the aforementioned involution $\phi:\fS_n\to \fS_n$ and verify that for every $\sigma\in\fS_n$, we have
\begin{align}\label{id:phi-p1-p2}
p_1(\sigma) = p_2(\phi(\sigma)).
\end{align}

The mapping $\phi$ is defined with respect to three cases. If $\sigma$ avoids both $p_1$ and $p_2$, we define $\phi(\sigma):=\sigma$ and it clearly satisfies \eqref{id:phi-p1-p2}. If $\sigma$ contains both $p_1$ and $p_2$, we also define $\phi(\sigma):=\sigma$. Thanks to Lemma~\ref{lem:same-a}, we see that in this case $p_1(\sigma)=p_2(\sigma)$, hence identity \eqref{id:phi-p1-p2} still holds true. For the final case, relying on Lemma~\ref{lem:unique-a} we can assume that 
$$\sigma\in \fS_n(p_1) \mathbin{\triangle} \fS_n(p_2):= \left(\fS_n(p_2)\setminus\fS_n(p_1)\right)\cup \left(\fS_n(p_1)\setminus\fS_n(p_2)\right)$$ 
with $a:=\anc(\sigma)$ and $k:=\sigma^{-1}(a)>1$\footnote{If $k=1$ then $(a,b)$ is an occurrence of $p_1$ if and only if it is an occurrence of $p_2$, and thus $\sigma$ cannot be in the symmetric difference $\fS_n(p_1) \mathbin{\triangle} \fS_n(p_2)$.}. For $p=p_1$ or $p_2$, we let
$$\Omega(p;\sigma):=\{k<j\le n:\text{$(a,\sigma_j)$ constitutes an occurrence of $p$ in $\sigma$}\}$$
denote the set of the positions that correspond to the values $b$ such that $(a,b)$ forms an occurrence of pattern $p$ in $\sigma$, and let 
\begin{align*}
\mla(\sigma):= & \text{ the index of the minimal element among}\\
& \text{ all elements strictly to the left of $a$ in $\sigma$.}
\end{align*} 
Assuming $\sigma_{\mla(\sigma)}=m$, we construct the image permutation $\tau$ according to the following two subcases:
% $$\mla(\sigma):=\min\{j\in [1,k-1]: \sigma_l\ge \sigma_j \text{ for all $1\le l<k$}\}$$ That is, the prefix of $\sigma$ to the left of $a$ attains its minimum at position $\mla(\sigma)$.
\begin{itemize}
    \item \textbf{Case 1:} If $m > a$, then we see that $\sigma\in\fS_n(p_2)\setminus\fS_n(p_1)$, and we make the decomposition $[n]=[1,a)\cup[a,m)\cup[m,n]=:V_1\cup V_2\cup V_3$; see the left diagram in Figure~\ref{fig:class54-example}. Let $\tau=\tau_1\ldots \tau_n$ be the permutation such that for each $1\le i\le n$
    \begin{align*}
    \tau_i &=\begin{cases}
    \sigma_i & \text{if $\sigma_i\in V_1$,}\\
    \sigma_i+(n+1-m) & \text{if $\sigma_i\in V_2$,}\\
    \sigma_i-(m-a) & \text{if $\sigma_i\in V_3$.}
    \end{cases}
    \end{align*}
    % We increase each element whose value is in $[a, m)$ by $n - m + 1$, and decrease each element whose value is in $[m, n]$ by $m - a$. 
    \item \textbf{Case 2:} If $m < a$, then we see that $\sigma\in\fS_n(p_1)\setminus\fS_n(p_2)$, and we make the decomposition $[n]=[1,m)\cup[m,a)\cup[a,n]=:V_1\cup V_2\cup V_3$; see the right diagram in Figure~\ref{fig:class54-example}. Let $\tau=\tau_1\ldots \tau_n$ be the permutation such that for each $1\le i\le n$
    \begin{align*}
    \tau_i &=\begin{cases}
    \sigma_i & \text{if $\sigma_i\in V_1$,}\\
    \sigma_i+(n+1-a) & \text{if $\sigma_i\in V_2$,}\\
    \sigma_i-(a-m) & \text{if $\sigma_i\in V_3$.}
    \end{cases}
    \end{align*}
\end{itemize}

We claim that our mapping $\phi$ sends $\sigma$ to a permutation $\tau$ such that
\begin{align}
  \anc(\sigma) &= \anc(\tau),\label{eq:same anc}\\
  \mla(\sigma) &= \mla(\tau),\label{eq:same mla}\\
  \Omega(p_1;\sigma) &= \Omega(p_2;\tau),\label{eq:Omega-1}\\
  \Omega(p_2;\sigma) &= \Omega(p_1;\tau).\label{eq:Omega-2}
\end{align}
Note that \eqref{eq:Omega-1} already implies \eqref{id:phi-p1-p2}. It is fairly routine to verify all of the claimed relations (9)--(12). We give the details for \eqref{eq:Omega-1} and leave the remaining verifications to the reader. Recall that $\anc(\sigma)=\anc(\tau)=a$, $\sigma_{\mla(\sigma)}=m$. Now suppose $j\in\Omega(p_1;\sigma)$ so $(a,\sigma_j)$ forms a $p_1$-pattern in $\sigma$. By the restriction of $p_1$, we must have $\sigma_j<m$, and all elements to the right of $\sigma_j$ have values inside the interval $[a,\sigma_j]$. So these elements (including $a$ and $\sigma_j$ themselves) all belong to $V_2=[a,m)$, ensuring that their relative orders in $\tau$ remain the same as in $\sigma$. Consequently, $(a+(n+1-m),\tau_j)$ forms a $p_2$-occurrence in $\tau$ hence $j\in\Omega(p_2;\tau)$. Conversely, every $p_2$-occurrence, say $(a',\tau_j)$, in $\tau$ must correspond to a $p_1$-occurrence, namely $(a'-(n+1-m),\sigma_j)$, in $\sigma$, showing that if $j\in\Omega(p_2;\tau)$ then $j\in\Omega(p_1;\sigma)$. Such a one-to-one correspondence proves \eqref{eq:Omega-1}.

Suppose $\sigma$ is in Case 1 so that $\Omega(p_1;\sigma)\neq\varnothing$ and $\Omega(p_2;\sigma)=\varnothing$, applying \eqref{eq:Omega-1} and \eqref{eq:Omega-2} we see that $\tau$ is in Case 2. A similar argument proves that if $\sigma$ is in Case 2 then $\tau$ is in Case 1. From the construction of $\phi$ as well as the relations \eqref{eq:same anc} and \eqref{eq:same mla}, one sees that $\phi(\phi(\sigma))=\sigma$ for every $\sigma\in \fS_n(p_1) \mathbin{\triangle} \fS_n(p_2)$, hence $\phi$ is indeed an involution in this third and final case.
\end{proof}

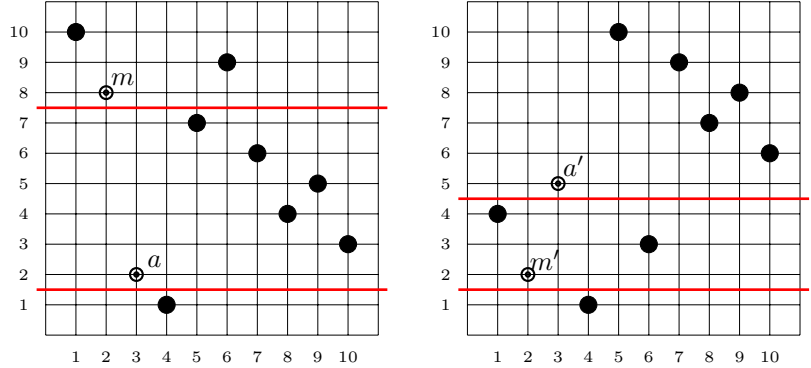
\begin{figure}[!h]  
  \begin{center}  
    
    \begin{tabular}{cc}

      \begin{tikzpicture}[scale=0.40] 
        \tikzset{    
          grid/.style={ draw, step=1cm, black!100, very thin }, 
          cell/.style={ draw, anchor=center, text centered },  
          graycell/.style={ fill=gray!40, draw=none, minimum width=1cm, minimum height=1cm, anchor=south west } 
        }  
        
        \draw[grid] (0,0) grid (11,11);
        
        % 坐\UTF{6807}\UTF{8F74}刻度数字
        \foreach \x in {1,...,10} \node[anchor=north,font=\tiny] at (\x,-0.2) {\x};
        \foreach \y in {1,...,10} \node[anchor=east,font=\tiny] at (-0.2,\y) {\y};
        
        % 粗体方框
        \draw[line width=1pt, red] (-0.3,7.5) -- (11.3,7.5);
        \draw[line width=1pt, red] (-0.3,1.5) -- (11.3,1.5);
        
        % 散点
        \filldraw[black] (1,10) circle (8pt);
        %\filldraw[black] (2,2) circle (8pt);
        %\filldraw[black] (3,4) circle (8pt);
        \filldraw[black] (4,1) circle (8pt);
        \filldraw[black] (5,7) circle (8pt);
        \filldraw[black] (6,9) circle (8pt);
        \filldraw[black] (7,6) circle (8pt);
        \draw[thick] (3,2) circle (6pt); \draw[thick] (3,2) circle (2pt); \node[above right] at (3, 1.9) {\footnotesize $a$};
        %\draw[thick] (1,3) circle (8pt); \draw[thick] (1,3) circle (2pt);
        \draw[thick] (2,8) circle (6pt); \draw[thick] (2,8) circle (2pt); \node[above right] at (1.8, 7.9) {\footnotesize $m$};
        \filldraw[black] (8,4) circle (8pt);
        \filldraw[black] (9,5) circle (8pt);
        \filldraw[black] (10,3) circle (8pt);
        %\draw[thick] (10,4) circle (8pt); \draw[thick] (10,5) circle (2pt);  
        
      \end{tikzpicture}
      
      &
      
      \begin{tikzpicture}[scale=0.40] 
        \tikzset{    
          grid/.style={ draw, step=1cm, black!100, very thin }, 
          cell/.style={ draw, anchor=center, text centered },  
          graycell/.style={ fill=gray!40, draw=none, minimum width=1cm, minimum height=1cm, anchor=south west } 
        }  
        
        \draw[grid] (0,0) grid (11,11);
        
        % 坐\UTF{6807}\UTF{8F74}刻度数字
        \foreach \x in {1,...,10} \node[anchor=north,font=\tiny] at (\x,-0.2) {\x};
        \foreach \y in {1,...,10} \node[anchor=east,font=\tiny] at (-0.2,\y) {\y};
        
        % 粗体方框
        \draw[line width=1pt, red] (-0.3,4.5) -- (11.3,4.5);
        \draw[line width=1pt, red] (-0.3,1.5) -- (11.3,1.5);

        % 散点
        \filldraw[black] (1,4) circle (8pt);
        %\filldraw[black] (2,2) circle (8pt);
        %\filldraw[black] (3,4) circle (8pt);
        \filldraw[black] (4,1) circle (8pt);
        \filldraw[black] (5,10) circle (8pt);
        \filldraw[black] (6,3) circle (8pt);
        \filldraw[black] (7,9) circle (8pt);
        \draw[thick] (3,5) circle (6pt); \draw[thick] (3,5) circle (2pt); \node[above right] at (2.8, 4.9) {\footnotesize $a'$};
        %\draw[thick] (1,3) circle (8pt); \draw[thick] (1,3) circle (2pt);
        \draw[thick] (2,2) circle (6pt); \draw[thick] (2,2) circle (2pt); \node[above right] at (1.7, 1.8) {\footnotesize $m'$};
        \filldraw[black] (8,7) circle (8pt);
        \filldraw[black] (9,8) circle (8pt);
        \filldraw[black] (10,6) circle (8pt);
        %\draw[thick] (10,4) circle (8pt); \draw[thick] (10,5) circle (2pt);        
      \end{tikzpicture}
      
    \end{tabular}
  \end{center}
  \caption{The permutation $\sigma=(10)821796453$ with $\Omega(p_1;\sigma)=\{7,10\}$ is mapped to $\tau=4251(10)39786$ with $\Omega(p_2;\tau)=\{7,10\}$ under the involution $\phi$.}\label{fig:class54-example}
\end{figure}

For a word $w$, any subword $u$ consisting of contiguous letters in $w$ is called a \emph{factor} of $w$, and we denote by $w^{\rr}$ the \emph{reverse} of $w$, i.e., the word obtained from $w$ by reversing the order of its letters. For a word $w$ and a letter $x$, sometimes we shall abuse the notation to write $x>w$, meaning that $x$ is greater than every letter in $w$. The inequality $x<w$ is understood analogously. Recall that a \emph{right-to-left maximum} in a given permutation $\sigma=\sigma_1\ldots\sigma_n$ refers to a letter $\sigma_i$, such that $\sigma_i>\sigma_j$ for all $i<j\le n$. Let $\RLM(\sigma)$ denote the set of right-to-left maxima in $\sigma$.

Next we deal with Class 71. It consists of $6$ patterns divided into two subclasses (one with $2$ patterns and the other with $4$ patterns); see \cite[Tab.~5]{SKZ2026} and \cite[Tab.~2, Classes 61 and 62]{KZ2019}. We take $p_3:=\pattern{scale=0.6}{2}{1/1,2/2}{0/0,1/2,2/2,2/1}$ from the first subclass and $p_4:=\pattern{scale=0.6}{2}{1/1,2/2}{0/2,1/0,2/1,2/2}$ from the second subclass, and show the following result, thereby confirming one of the conjectured equidistributions from \cite[Tab.~2]{KZ2019}. 

\begin{thm}\label{thm:class71-equi}
There exists a bijection $\psi:\fS_n\to\fS_n$ such that
\begin{align}
\RLM(\sigma) &=\RLM(\psi(\sigma)),\label{id:RLMax}\\
\label{id:p3-p4}
p_3(\sigma) &=p_4(\psi(\sigma)),
\end{align} 
for every $\sigma\in\fS_n$ and $n\ge 1$. In particular, identity \eqref{id:p3-p4} implies that $p_3\sim p_4$, thus the $6$ patterns in Class $71$ are all equidistributed.
\end{thm}

\begin{proof}
  Given a permutation $\sigma\in\fS_n$ with $\RLM(\sigma)=\{r_1(=n),r_2,\ldots,r_l(=\sigma_n)\}_{>}$, we decompose it uniquely as
  $$\sigma = N_1 M_1 r_1 N_2 M_2 r_2\dots N_{l-1}M_{l-1}r_{l-1}N_l r_l,$$
  where for every $1\le i\le l-1$, $M_i$ is the longest factor whose letters, if any, are all less than $r_{i+1}$. We define its image $\tau=\psi(\sigma)$ to be
  $$\tau:=N^{\rr}_l N^{\rr}_{l-1} \dots N^{\rr}_1r_1 M_1 r_2 M_2\dots r_{l-1}M_{l-1} r_l.$$
The reader may refer to Figure~\ref{fig:class71-decomp} for an illustration of the above decompositions for generic permutations $\sigma$ and $\tau$.

  \begin{figure}[!h]
  \begin{center}
    \begin{tabular}{cc}
      \begin{tikzpicture}[scale=0.7]
        \tikzset{
          grid/.style={ draw, step=1cm, black!100, thin },
          graycell/.style={ fill=gray!50, draw=none, minimum width=1cm, minimum height=1cm, anchor=south west }
        }
        
        \fill[graycell] (1,5) rectangle (7,7);
        \fill[graycell] (3,3) rectangle (7,5);
        \fill[graycell] (5,1) rectangle (7,3);
        \draw (0,0) rectangle (7,7);

        \filldraw[black] (2,7) circle (2.5pt);
        \node[anchor=west] at (1.9,6.7) {\footnotesize{$r_1=n$}};
        
        \filldraw[black] (4,5) circle (2.5pt);
        \node[anchor=west] at (3.9,4.7) {\footnotesize{$r_2$}};
        
        \filldraw[black] (6,3) circle (2.5pt);
        \node[anchor=west] at (5.9,2.7) {\footnotesize{$r_3$}};
        
        \filldraw[black] (7,1) circle (2.5pt);
        \node[anchor=west] at (6.9,0.7) {\footnotesize{$r_4=\sigma_n$}};
        
        \draw[thick] (0.1, 0.1) rectangle (0.9, 6.9);
        \node at (0.5, 3.5) {$N_1$};  
        \draw[thick] (1.1, 0.1) rectangle (1.9, 4.9);
        \node at (1.5, 2.5) {$M_1$};    
        \draw[thick] (2.1, 0.1) rectangle (2.9, 4.9);
        \node at (2.5, 2.5) {$N_2$};    
        \draw[thick] (3.1, 0.1) rectangle (3.9, 2.9);
        \node at (3.5, 1.5) {$M_2$};
        \draw[thick] (3.1, 0.1) rectangle (3.9, 2.9);
        \node at (3.5, 1.5) {$M_2$};
        \draw[thick] (4.1, 0.1) rectangle (4.9, 2.9);
        \node at (4.5, 1.5) {$N_3$};
        \draw[thick] (5.1, 0.1) rectangle (5.9, 0.9);
        \node at (5.5, 0.5) {$M_3$};
        \draw[thick] (6.1, 0.1) rectangle (6.9, 0.9);
        \node at (6.5, 0.5) {$N_4$};        
      \end{tikzpicture}
      &
      \begin{tikzpicture}[scale=0.7]
        \tikzset{
          grid/.style={ draw, step=1cm, black!100, thin },
          graycell/.style={ fill=gray!50, draw=none, minimum width=1cm, minimum height=1cm, anchor=south west }
        }
        
        \fill[graycell] (0,5) rectangle (3,7);
        \fill[graycell] (4,5) rectangle (7,7);
        
        \fill[graycell] (0,3) rectangle (2,5);
        \fill[graycell] (5,3) rectangle (7,5);
        
        \fill[graycell] (0,1) rectangle (1,3);
        \fill[graycell] (6,1) rectangle (7,3);
        
        \draw (0,0) rectangle (7,7);

        \filldraw[black] (4,7) circle (2.5pt);
        \node[anchor=west] at (3.9,6.7) {\footnotesize{$r_1=n$}};
        
        \filldraw[black] (5,5) circle (2.5pt);
        \node[anchor=west] at (4.9,4.7) {\footnotesize{$r_2$}};
        
        \filldraw[black] (6,3) circle (2.5pt);
        \node[anchor=west] at (5.9,2.7) {\footnotesize{$r_3$}};
        
        \filldraw[black] (7,1) circle (2.5pt);
        \node[anchor=west] at (6.9,0.7) {\footnotesize{$r_4=\tau_n$}};
        
        \draw[thick] (0.1, 0.1) rectangle (0.9, 0.9);
        \node at (0.5, 0.5) {$N^{\rr}_4$};  
        \draw[thick] (1.1, 0.1) rectangle (1.9, 2.9);
        \node at (1.5, 1.5) {$N^{\rr}_3$};    
        \draw[thick] (2.1, 0.1) rectangle (2.9, 4.9);
        \node at (2.5, 2.5) {$N^{\rr}_2$};    
        \draw[thick] (3.1, 0.1) rectangle (3.9, 6.9);
        \node at (3.5, 3.5) {$N^{\rr}_1$};
        
        \draw[thick] (4.1, 0.1) rectangle (4.9, 4.9);
        \node at (4.5, 2.5) {$M_1$};
        \draw[thick] (5.1, 0.1) rectangle (5.9, 2.9);
        \node at (5.5, 1.5) {$M_2$};
        \draw[thick] (6.1, 0.1) rectangle (6.9, 0.9);
        \node at (6.5, 0.5) {$M_3$};
        
      \end{tikzpicture}
    \end{tabular}
    \caption{The decompositions of generic $\sigma$ and $\tau=\psi(\sigma)$.} \label{fig:class71-decomp}
  \end{center}
\end{figure}
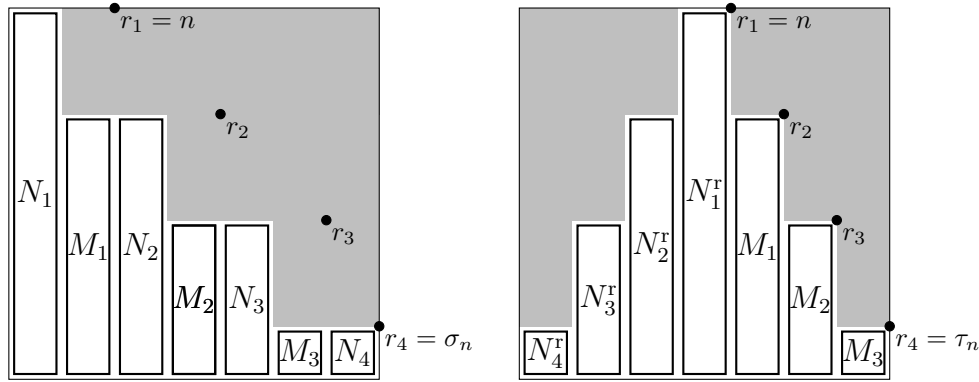

  It is true that $N^{\rr}_l N^{\rr}_{l-1} \dots N^{\rr}_1=(N_1\dots N_{l-1}N_l)^{\rr}$. However, we prefer the former expression for $\tau$, as viewing it as a block-by-block reversal is useful for constructing the inverse mapping $\psi^{-1}$. Indeed, by our definition of $M_i$'s we see that $\RLM(\tau)=\{r_1,r_2,\ldots,r_l\}$, which proves \eqref{id:RLMax}. Moreover, if $N_i^{\rr}$ is nonempty for some $1\le i\le l$, then its leftmost element (or equivalently, the rightmost element of $N_i$), say $n_i$, must satisfy that\footnote{Here we set $r_{l+1}:=0$ as a convention.} 
  \begin{enumerate}[label=(\arabic*)]
    \item $r_{i+1}<n_i<r_i$; and 
    \item all elements to the left of $n_i$ are less than $r_{i+1}$.
  \end{enumerate} 
  These observations result in a decomposition of the prefix (to the left of $r_1=n$ in $\tau$) into blocks, so that one can utilize $(N^{\rr}_i)^{\rr}=N_i$ to uniquely recover the preimage $\psi^{-1}(\tau)$. Therefore $\psi^{-1}$ is well-defined and $\psi$ is indeed a bijection.

  It remains to verify that $\psi$ has property \eqref{id:p3-p4}. To that end, we require some further observations on $\sigma$ and $\tau$. Suppose $(a,b)$ is an occurrence of the pattern $p_3$ in $\sigma$, and $(a',b')$ is an occurrence of the pattern $p_4$ in $\tau$. We claim that
\begin{enumerate}[label=(\arabic*), resume]
  \item $b$ (resp.~$b'$) must be a right-to-left maximum of $\sigma$ (resp.~$\tau$), say $b=r_j$ (resp.~$b'=r_k$).
  \item $a\in N_j$ so that we can decompose $N_j=P_j a Q_j$ for some (possibly empty) factors $P_j$ and $Q_j$.
  \item $a>r_{j+1}$ and $a<N_1M_1\ldots N_{j-1}M_{j-1}P_j$. 
  \item $a'\in N^{\rr}_k$ so that we can decompose $N^{\rr}_k=Q^{\rr}_k a' P^{\rr}_k$ for some (possibly empty) factors $Q^{\rr}_k$ and $P^{\rr}_k$.
  \item $a'>r_{k+1}$ and $a< P^{\rr}_k N^{\rr}_{j-1}\ldots N^{\rr}_{1} M_1\ldots M_{k-1}$.
\end{enumerate}
  
  Crucial to the verification of all above claims is a precise grasp of the constraints enforced by the shadings in the mesh patterns $p_3$ and $p_4$. More precisely, the northeast corner shadings in both $p_3$ and $p_4$ readily give us (3). To see (4) and (5), note that we have assumed in (3) that $b=r_j$, so $a$ must belong to the subword $N_1M_1\ldots N_jM_j$. However, when $j>1$, the presence of $r_{j-1}$ and the shaded $(1,2)$-cell in $p_3$ will prevent any letter belonging to $N_1M_1\ldots N_{j-1}M_{j-1}$ from making a valid $p_3$-pattern with $b=r_j$, hence $a\in N_jM_j$. In addition, if $a<r_{j+1}$ then $r_{j+1}$, occupying the shaded $(2,1)$-cell in $p_3$, will render $(a,b)=(a,r_j)$ an invalid $p_3$-pattern. Therefore $a>r_{j+1}$ as claimed. We further use the shaded $(0,0)$-cell in $p_3$ to deduce that $a<N_1M_1\ldots N_{j-1}M_{j-1}P_j$ and confirm all claims in (4) and (5). Claims in (6) and (7) follow from a mostly similar argument that leverages the shaded regions in pattern $p_4$. In particular, we emphasize that to see that $a'$ cannot belong to $N^{\rr}_{\ell}$ for a certain $1\le \ell <k$, one needs the previous claim (1) saying that the leftmost letter, say $n_{\ell}$, in $N^{\rr}_{\ell}$ must satisfy $n_{\ell}>r_{\ell+1}\ge r_k=b'>a'$, which turns $(a',b')$ into an invalid $p_4$-pattern, a contradiction. Further details are omitted. 

  Finally we notice that property \eqref{id:p3-p4} follows from the following stronger fact:
\begin{enumerate}[label=(\arabic*), resume]
   \item $(a,b)$ forms a $p_3$-pattern in $\sigma$ (i.e. it satisfies (3), (4), and (5)), if and only if $(a,b)$ forms a $p_4$-pattern in the image $\tau=\psi(\sigma)$ (i.e., it satisfies (3), (6), and (7)).
\end{enumerate}
The validity of (8) follows directly from the evidently parallel formulations of (4)--(5) and (6)--(7).
\end{proof}

\begin{example}\label{eg:psi}
	Take $\sigma= 14~12~10~\mathbf{15}~11~9~\mathbf{13}~7~2~6~4~3~\mathbf{8}~1~\mathbf{5}$ with $\RLM(\sigma)=\{15,13,8,5\}$. We have $N_1=14$, $M_1=12\,10$, $N_2=11\,9$, $M_2=\varnothing$, $N_3=7\,2\,6$, $M_3=4\,3$, and $N_4=1$. Thus,
  \begin{align*}
  \tau &:=\psi(\sigma) = N^{\rr}_4\,N^{\rr}_3\,N^{\rr}_2\,N^{\rr}_1\,\mathbf{15}\,M_1\,\mathbf{13}\,M_2\,\mathbf{8}\,M_3\,\mathbf{5}\\
  &=1~6~2~7~9~11~14~\mathbf{15}~12~10~\mathbf{13}~\mathbf{8}~4~3~\mathbf{5}.
  \end{align*} 
  Agreeing with the claim (8) in the proof of Theorem~\ref{thm:class71-equi}, the pairs $(14,15), (9,13), (7,8)$, and $(1,5)$ are simultaneously the occurrences of $p_3$ in $\sigma$ and the occurrences of $p_4$ in $\tau$; see Figure~\ref{fig:class71-example} below for an illustration.
\end{example}
 % The left diagram shows $\sigma$ decomposed into blocks $N_i$ and $M_i$ relative to the right-to-left maxima $r_i$. While the right diagram shows the image permutation $\tau=\psi(\sigma)$, where blocks $M_i$ maintain their internal structures while blocks $N_i$ are reversed to $N^{\rr}_i$ and shifted to the prefix.
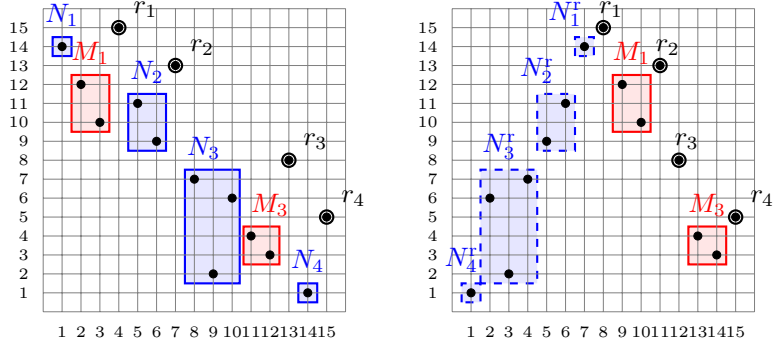
\begin{figure}[!h]  
	\begin{center}  
		\begin{tabular}{cc}
			
			% ==========================================
			% 左图：原排列 \pi
			% ==========================================
			\begin{tikzpicture}[scale=0.25] 
				\tikzset{    
					grid/.style={ draw, step=1cm, black!50, very thin }, 
					nblock/.style={ fill=blue!10, draw=blue, thick }, % N_i 块样式 (浅蓝底, 蓝框)
					mblock/.style={ fill=red!10, draw=red, thick }    % M_i 块样式 (浅红底, 红框)
				}  
				
				% 1. 绘制色块底色 (在网格下方)
				\filldraw[nblock] (0.5, 13.5) rectangle (1.5, 14.5);   % N_1
				\filldraw[mblock] (1.5, 9.5)  rectangle (3.5, 12.5);   % M_1
				\filldraw[nblock] (4.5, 8.5)  rectangle (6.5, 11.5);   % N_2
				\filldraw[nblock] (7.5, 1.5)  rectangle (10.4, 7.5);    % N_3
				\filldraw[mblock] (10.6, 2.5)  rectangle (12.5, 4.5);   % M_3
				\filldraw[nblock] (13.5, 0.5) rectangle (14.5, 1.5);   % N_4
				
				% 2. 绘制网格
				\draw[grid] (0,0) grid (16,16);
				
				% 3. 坐标轴刻度数字
				\foreach \x in {1,...,15} \node[anchor=north,font=\tiny] at (\x,-0.2) {\x};
				\foreach \y in {1,...,15} \node[anchor=east,font=\tiny] at (-0.2,\y) {\y};
				
				% 4. 绘制块标签
				\node[blue, above] at (1, 14.5) {\footnotesize $N_1$};
				\node[red, above]  at (2.5, 12.5) {\footnotesize $M_1$};
				\node[blue, above] at (5.5, 11.5) {\footnotesize $N_2$};
				\node[blue, above] at (8.5, 7.5)  {\footnotesize $N_3$};
				\node[red, above]  at (12, 4.5)  {\footnotesize $M_3$};
				\node[blue, above] at (14, 1.5)  {\footnotesize $N_4$};
				
				% 5. 绘制散点
				% N_1
				\filldraw[black] (1,14) circle (6pt);
				% M_1
				\filldraw[black] (2,12) circle (6pt); \filldraw[black] (3,10) circle (6pt);
				% r_1
				\filldraw[black] (4,15) circle (6pt); \draw[thick] (4,15) circle (10pt); \node[above right] at (4.2, 15) {\footnotesize $r_1$};
				% N_2
				\filldraw[black] (5,11) circle (6pt); \filldraw[black] (6,9) circle (6pt);
				% r_2
				\filldraw[black] (7,13) circle (6pt); \draw[thick] (7,13) circle (10pt); \node[above right] at (7.2, 13) {\footnotesize $r_2$};
				% N_3
				\filldraw[black] (8,7) circle (6pt); \filldraw[black] (9,2) circle (6pt);
				% M_3
				\filldraw[black] (10,6) circle (6pt); \filldraw[black] (11,4) circle (6pt); \filldraw[black] (12,3) circle (6pt);
				% r_3
				\filldraw[black] (13,8) circle (6pt); \draw[thick] (13,8) circle (10pt); \node[above right] at (13.2, 8) {\footnotesize $r_3$};
				% N_4
				\filldraw[black] (14,1) circle (6pt);
				% r_4
				\filldraw[black] (15,5) circle (6pt); \draw[thick] (15,5) circle (10pt); \node[above right] at (15.2, 5) {\footnotesize $r_4$};
			\end{tikzpicture}
			
			&
			
			% ==========================================
			% 右图：目标排列 \sigma
			% ==========================================
			\begin{tikzpicture}[scale=0.25] 
				\tikzset{    
					grid/.style={ draw, step=1cm, black!50, very thin }, 
					nblock/.style={ fill=blue!10, draw=blue, thick, dashed }, % N^{\rr}_i 用虚线框表示反转
					mblock/.style={ fill=red!10, draw=red, thick }            % M_i 块样式保持不变
				}  
				
				% 1. 绘制色块底色 (注意：这里的 N_i 变成了逆序排列，且移动到了最前面)
				\filldraw[nblock] (0.5, 0.5)  rectangle (1.5, 1.5);    % N^{\rr}_4
				\filldraw[nblock] (1.5, 1.5)  rectangle (4.5, 7.5);    % N^{\rr}_3
				\filldraw[nblock] (4.5, 8.5)  rectangle (6.5, 11.5);   % N^{\rr}_2
				\filldraw[nblock] (6.5, 13.5) rectangle (7.5, 14.5);   % N^{\rr}_1
				\filldraw[mblock] (8.5, 9.5)  rectangle (10.5, 12.5);   % M_1
				\filldraw[mblock] (12.5, 2.5) rectangle (14.5, 4.5);   % M_3
				
				% 2. 绘制网格
				\draw[grid] (0,0) grid (16,16);
				
				% 3. 坐标轴刻度数字
				\foreach \x in {1,...,15} \node[anchor=north,font=\tiny] at (\x,-0.2) {\x};
				\foreach \y in {1,...,15} \node[anchor=east,font=\tiny] at (-0.2,\y) {\y};
				
				% 4. 绘制块标签
				\node[blue, above] at (0.5, 1.5)   {\footnotesize $N^{\rr}_4$};
				\node[blue, above] at (2.5, 7.5) {\footnotesize $N^{\rr}_3$};
				\node[blue, above] at (4.5, 11.5){\footnotesize $N^{\rr}_2$};
				\node[blue, above] at (6, 14.5)  {\footnotesize $N^{\rr}_1$};
				\node[red, above]  at (9.5, 12.5) {\footnotesize $M_1$};
				\node[red, above]  at (13.5, 4.5)  {\footnotesize $M_3$};
				
				% 5. 绘制散点
				% N^{\rr}_4
				\filldraw[black] (1,1) circle (6pt);
				% N^{\rr}_3
				\filldraw[black] (2,6) circle (6pt); \filldraw[black] (3,2) circle (6pt);
				% N^{\rr}_2
				\filldraw[black] (4,7) circle (6pt); \filldraw[black] (5,9) circle (6pt);
				% N^{\rr}_1
				\filldraw[black] (6,11) circle (6pt);
				% r_1
				\filldraw[black] (7,14) circle (6pt);\filldraw[black] (8,15) circle (6pt); \draw[thick] (8,15) circle (10pt); \node[above right] at (7.2, 15) {\footnotesize $r_1$};
				% M_1
				\filldraw[black] (9,12) circle (6pt); \filldraw[black] (10,10) circle (6pt);
				% r_2
				\filldraw[black] (11,13) circle (6pt); \draw[thick] (11,13) circle (10pt); \node[above right] at (10.2, 13) {\footnotesize $r_2$};
				% r_3
				\filldraw[black] (12,8) circle (6pt); \draw[thick] (12,8) circle (10pt); \node[above right] at (11.2, 8) {\footnotesize $r_3$};
				% M_3
			   \filldraw[black] (13,4) circle (6pt); \filldraw[black] (14,3) circle (6pt);
				% r_4
				\filldraw[black] (15,5) circle (6pt); \draw[thick] (15,5) circle (10pt); \node[above right] at (15.2, 5) {\footnotesize $r_4$};
				
			\end{tikzpicture}
			
		\end{tabular}
	\end{center}
	\caption{The decompositions for $\sigma$ and $\tau=\psi(\sigma)$ in Example~\ref{eg:psi}.}\label{fig:class71-example}
\end{figure}

%%%%%%%%%%%%%%%%%%%%%%%%%%%%%%%%%%%%%%%%%%%%%%%%%%%%%%%%%%%%%%%%%%%%%%%%%%%%%%%%%%%%%%%%%
\section{Proofs of the three distributions}\label{sec:proofs_of_the_three_distributions}

Fix a pattern $p$, recall the bivariate generating function $F^p(t,q)=\sum_{n,k\ge 0}s_{n,k}(p)q^k t^n$, and its two specializations $A^p(t)=F^p(t,0)$ and $F(t)=F^p(t,1)=\sum_{n\ge 0}n! t^n$. We begin with a useful lemma concerning the mesh pattern of length $1$.

\begin{lem}[Thm. 1.1 in \cite{KZ2019}]\label{lem:short} 
For the pattern $p:=\pattern{scale=0.6}{1}{1/1}{0/1,1/0}$ (equivalently, $p:=\pattern{scale=0.6}{1}{1/1}{0/0,1/1}$), we have
$$A^p(t)=\frac{F(t)}{1+tF(t)}:=\sum_{n\ge 0}S_nt^n;\quad F^p(t,q)=\frac{F(t)}{1+t(1-q)F(t)}:=\sum_{n\ge 0}T_n(q)t^n.$$
\end{lem}

\begin{proof}[Proof of Theorem~\ref{thm:class54}]
It is sufficient to justify the statement for the pattern $p_5:=\pattern{scale=0.6}{2}{1/1,2/2}{0/2,2/2,0/1,1/1,2/0}$. 
% \iffalse
% Given a permutation $\pi=\pi_1\cdots \pi_n\in \fS_n$ with $p(\pi)\geq 1$, that is, pattern $p$ occurs at least once in $\pi$. 
% \fi

Let $\pi=\pi_1\cdots \pi_n\in \fS_n$ be a permutation with $p_5(\pi)\geq 1$, meaning that the pattern $p_5$ occurs at least once in $\pi$. Let $b$ be the leftmost element such that there exists a certain $a<b$ and $(a,b)$ forms a $p_5$-pattern in $\pi$. Note that once $b$ is fixed, its partner $a$ is actually uniquely determined, due to the two shades at cells $(0,1)$ and $(1,1)$ in the pattern $p_5$. With respect to $(a,b)$, we can decompose the permutation as shown in the left diagram of Figure \ref{fig:class54}. Our discussion splits into the following two cases, depending on whether the elements in block $C$ avoid the pattern $\pattern{scale=0.6}{1}{1/1}{0/0,1/1}$.
\begin{itemize}
\item If block $C$ avoids the pattern $\pattern{scale=0.6}{1}{1/1}{0/0,1/1}$, then we assume the four blocks $A,B,C,D$ (as depicted in Figure~\ref{fig:class54}) contain $i,j,k,m$ elements, respectively, subject to the constraint $i+j+k+m=n-2$. 

Firstly, we note that an occurrence, say $d$, of pattern $\pattern{scale=0.6}{1}{1/1}{0/0,1/1}$ in block $D$, is in one-to-one correspondence with an occurrence $(a,d)$ of pattern $p_5$ in $\pi$. And occurrences of $p_5$ in $D$ become non-occurrences in $\pi$ due to the presence of $b$. Consequently, the $m$ elements in block $D$ contribute $T_m(t)$ to the generating function by Lemma~\ref{lem:short}. 

Secondly, for blocks $B$ and $C$, any occurrence of $p_5$ in $B$ (resp.~$C$) becomes a non-occurrence in $\pi$ due to the presence of $a$ (resp.~$b$), and no new $p_5$ occurrence can be formed with $a$ since we have assumed in this case that the $k$ elements in block $C$ avoid pattern $\pattern{scale=0.6}{1}{1/1}{0/0,1/1}$. Consequently, we can freely choose, in $\binom{j+k}{k}$ ways, the $k$ positions for the elements (all larger than $b$ in their values) in block $C$, and the remaining $j$ positions are determined as well for elements in block $B$. The total contribution is $\binom{j+k}{k}S_k$ by Lemma~\ref{lem:short}.

Lastly, for elements in blocks $A$ and $B$, their positions are already determined, and their values should all be smaller than $a$. There are no further restrictions, hence there are $(i+j)!$ ways to arrange them. Combining the above discussion, we see that the generating function in this case is given by 
\begin{align}\label{eq:C1}
\sum_{\substack{i, j, k, m\geq 0\\ i+j+k+m=n-2}}T_m(q) S_k \binom{j+k}{k}(i+j)!.
\end{align}

\item If block $C$ contains at least one occurrence of pattern $\pattern{scale=0.6}{1}{1/1}{0/0,1/1}$, then we select the rightmost such element $c$ and use it to further split the blocks $B$ and $C$ into smaller blocks $B_1,B_2$ and $C_1,C_2$; see the decomposition shown as the right diagram in Figure~\ref{fig:class54}, where the number of elements contained in each block is labeled out. In particular, notice that $B_2\neq \varnothing$ and we assume it contains $j_2+1$ elements with $j_2\ge 0$. Indeed, if $B_2=\varnothing$, then $(a,c)$ (with $c$ being to the left of $b$) would form a pattern $p_5$, which contradicts our choice of $b$. Also note the new constraint $i+j_1+j_2+k_1+k_2+m=n-4$. By a similar argument that uses Lemma~\ref{lem:short}, we see that the $m$ elements in block $D$ contribute $T_m(q)$ to the generating function.

For blocks $B_2$ and $C_2$, note that the elements in block $C_2$ avoid the pattern $\pattern{scale=0.6}{1}{1/1}{0/0,1/1}$ by the choice of $c$. Thus, the $k_2$ elements in block $C_2$ contribute $S_{k_2}$ by Lemma~\ref{lem:short}, and there are $\binom{k_2+j_2+1}{k_2}$ ways to distribute their positions over blocks $B_2$ and $C_2$. 

Next, there are $\binom{k_1+j_1}{k_1}$ ways to decide for blocks $B_1$ and $C_1$ the relative positions of the elements that they contain, and $k_1!$ ways to assign values (all larger than $c$) to those elements in block $C_1$.

Lastly, for the $i$ elements in block $A$, the $j_1$ elements in block $B_1$, and the $j_2+1$ elements in block $B_2$, their positions have all been determined. We need an extra factor $(i+j_1+j_2+1)!$ to account for the different ways of assigning values (all smaller than $a$) to these elements. Combining the above discussion, the generating function in this case is seen to be
\begin{align}\label{eq:C2}
\sum_{\substack{i, j_1, j_2, k_1, k_2, m\geq 0\\i+j_1+j_2+k_1+k_2+m=n-4}}T_m(q) S_{k_2} \binom{k_2+j_2+1}{k_2}\binom{k_1+j_1}{k_1}k_1!(i+j_1+j_2+1)!.
\end{align}
\end{itemize}
Finally, putting \eqref{eq:C1} and \eqref{eq:C2} together with $q$ (for the $p_5$-occurrence $(a,b)$) and $t^n$ (for the total number of elements), we arrive at~\eqref{class54-distr}.
\end{proof}

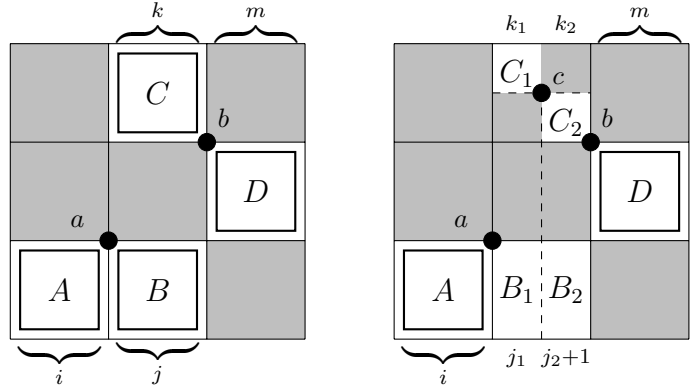
\begin{figure}
\begin{center}

\begin{tabular}{ccc}

\begin{tikzpicture}[scale=1.3]
  \tikzset{
    grid/.style={ draw, step=1cm, black!100, thin },
    graycell/.style={ fill=gray!50, draw=none, minimum width=1cm, minimum height=1cm, anchor=south west }
  }

  \fill[graycell] (0,2) rectangle (1,3);
  \fill[graycell] (2,2) rectangle (3,3);
  \fill[graycell] (1,1) rectangle (2,2);
  \fill[graycell] (0,1) rectangle (1,2);
  \fill[graycell] (2,0) rectangle (3,1);

  \draw[grid] (0,0) grid (3,3);

  \draw[thick] (1.1, 2.1) rectangle (1.9, 2.9);
  \draw[thick] (1.1, 0.1) rectangle (1.9, 0.9);
  \draw[thick] (0.1, 0.1) rectangle (0.9, 0.9);
  \draw[thick] (2.1, 1.1) rectangle (2.9, 1.9);

  \node[anchor=center] at (1.5, 2.5) {$C$};
  \node[anchor=center] at (0.5, 0.5) {$A$};
  \node[anchor=center] at (1.5, 0.5) {$B$};
  \node[anchor=center] at (2.5, 1.5) {$D$};

  \filldraw[black] (1,1) circle (2.5pt);
  \node[anchor=west] at (0.5,1.2) {\footnotesize{$a$}};
  \filldraw[black] (2,2) circle (2.5pt);
  \node[anchor=west] at (2,2.25) {\footnotesize{$b$}};
  \node[anchor=south] at (0.5,-0.6) {$\underbrace{\phantom{aaaaa}}_{i}$};
  \node[anchor=south] at (1.5,-0.6) {$\underbrace{\phantom{aaaaa}}_{j}$};
  \node[anchor=south] at (1.5,2.75) {$\overbrace{\phantom{aaaaa}}^{k}$};
  \node[anchor=south] at (2.5,2.75) {$\overbrace{\phantom{aaaaa}}^{m}$};
\end{tikzpicture}

&

\ \ \ 

&

\begin{tikzpicture}[scale=1.3]
  \tikzset{
    grid/.style={ draw, step=1cm, black!100, thin },
    graycell/.style={ fill=gray!50, draw=none, minimum width=1cm, minimum height=1cm, anchor=south west }
  }

  \fill[graycell] (0,2) rectangle (1,3);
  \fill[graycell] (2,2) rectangle (3,3);
  \fill[graycell] (1,1) rectangle (2,2);
  \fill[graycell] (0,1) rectangle (1,2);
  \fill[graycell] (2,0) rectangle (3,1);
  \fill[graycell] (1,2) rectangle (1.5,2.5);
  \fill[graycell] (1.5,2.5) rectangle (2,3);

  \draw[dashed] (1.5, 0) -- (1.5, 2.5);
  \draw[dashed] (1,2.5) -- (2,2.5);

  \draw[grid] (0,0) grid (3,3);

  %\draw[thick] (1.1, 2.1) rectangle (1.9, 2.9);
  %\draw[thick] (1.1, 0.1) rectangle (1.4, 0.9);
  \draw[thick] (0.1, 0.1) rectangle (0.9, 0.9);
  \draw[thick] (2.1, 1.1) rectangle (2.9, 1.9);

  \node[anchor=center] at (1.25, 2.7) {$C_1$};
  \node[anchor=center] at (1.75, 2.2) {$C_2$};
  \node[anchor=center] at (0.5, 0.5) {$A$};
  \node[anchor=center] at (1.25, 0.5) {$B_1$};
  \node[anchor=center] at (1.75, 0.5) {$B_2$};
  \node[anchor=center] at (2.5, 1.5) {$D$};

  \filldraw[black] (1,1) circle (2.5pt);
  \node[anchor=west] at (0.5,1.2) {\footnotesize{$a$}};
  \filldraw[black] (2,2) circle (2.5pt);
  \node[anchor=west] at (2,2.25) {\footnotesize{$b$}};
  \filldraw[black] (1.5,2.5) circle (2.5pt);
  \node[anchor=west] at (1.5,2.65) {\footnotesize{$c$}};
  \node[anchor=south] at (0.5,-0.6) {$\underbrace{\phantom{aaaaa}}_{i}$};
  \node[anchor=south, font=\scriptsize] at (1.25,-0.4) {$j_1$};
  \node[anchor=south, font=\scriptsize] at (1.75,-0.4) {$j_2{+}1$};
  \node[anchor=south, font=\scriptsize] at (1.25,3) {$k_1$};
  \node[anchor=south, font=\scriptsize] at (1.75,3) {$k_2$};
  \node[anchor=south] at (2.5,2.75) {$\overbrace{\phantom{aaaaa}}^{m}$};
\end{tikzpicture}

\end{tabular}
\caption{Two cases in the proof of Theorem~\ref{thm:class54}.}\label{fig:class54}
\end{center}
\end{figure}

Throughout the next proof, we use $p_6:=\pattern{scale=0.6}{2}{1/1,2/2}{0/0,0/1,0/2,1/2,2/1}$, one of the eight equidistributed patterns in Class 60.
\begin{proof}[Proof of Theorem~\ref{thm:class60}]
If $(a,b)$ is an occurrence of $p_6$ in a permutation $\sigma=\sigma_1\cdots\sigma_n\in \fS_n$, then due to the restrictions of $p_6$, we see that $a=\pi_1$ and $b>a$. Moreover, if we use $\alpha$ to denote the subword formed by all of the elements in $\sigma$ larger than $a$, then $b$ must be an occurrence of the pattern $p:=\pattern{scale=0.6}{1}{1/1}{0/1,1/0}$ in $\alpha$. Conversely, any such occurrence of $p$ in $\alpha$ combined with $a=\sigma_1$ will form a $p_6$-pattern in $\sigma$. Assuming there are $m$ elements in $\alpha$, where $0\leq m\leq n-1$, we can select the positions of these $m$ elements in $\binom{n-1}{m}$ ways and permute the remaining elements (all smaller than $a$ in value) in $(n-m-1)!$ ways. This analysis entails the calculation
\begin{align*}
[t^n] F^{p_6}(t,q) &= \sum_{m=0}^{n-1}\binom{n-1}{m}(n-m-1)![t^m]F^p(t,q)= \sum_{m=0}^{n-1}\frac{(n-1)!}{m!}T_m(q),
\end{align*}
where we have applied Lemma~\ref{lem:short} for $F^p(t,q)$.
\end{proof}

For our third and final Class 75, we not only prove Theorem~\ref{thm:class75}, but also derive a Riccati differential equation~\cite{Rei1972} for its bivariate generating function $F(t,q)$ (see Corollary~\ref{cor:riccati}). We select from Class 75 a fixed pattern $p_7:=\pattern{scale=0.6}{2}{1/1,2/2}{0/1,1/2,2/2}$ and require the next definition.

\begin{defn}\label{def:admissible pos}
Given a permutation $\sigma=\sigma_1\ldots\sigma_n$, for each $1\le i\le n$, we increase every letter in $\sigma$ by $1$ and insert a new $1$ at the position immediately to the left of $\sigma_i$ to obtain a new permutation denoted as $\sigma^{(i)}$. The position $i$ is said to be \emph{admissible}, if $213(\sigma^{(i)})=213(\sigma)$.
\end{defn}

% \begin{prop}
% For $n,k,h\ge 0$, the numbers $s_{n,k,h}(p_7)$ satisfy

% \end{prop}
% It is sufficient to justify the statement for the pattern
% $p=\pattern{scale=0.6}{2}{1/1,2/2}{0/1,1/2,2/2}$.

\begin{proof}[Proof of Theorem~\ref{thm:class75}]
The initial conditions for $s_{n,k}(p_7)$ and $s_{n,k,h}(p_7)$ are straightforward to verify. To explain (\ref{thm:class75-rec-1}), we consider generating permutations $\sigma \in \fS_n$ with $k$ occurrences of $p_7$ by inserting the new smallest element $1$ into a permutation $\sigma'\in \fS_{n-1}$ in one of $n$ possible positions. Note that inserting $1$ never decreases the number of occurrences of $p_7$ in $\sigma'$. However, if $1$ is inserted into an admissible position, then exactly one new occurrence of $p_7$ is created, formed by $1$ and the maximum element to its right. It follows that $s_{n,k}(p_7)$ is given by the sum of $n s_{n-1,k}(p_7)-\sum_{h=1}^{n-1} h \cdot s_{n-1,k,h}(p_7)$ and $\sum_{h=1}^{n-1} h \cdot s_{n-1,k-1,h}(p_7)$,
which yields (\ref{thm:class75-rec-1}). Also, note that (\ref{thm:class75-rec-2}) follows directly from the definitions.

To explain (\ref{thm:class75-rec-3}), we consider generating permutations $\sigma\in \fS_n$ with exactly $k$ occurrences of $p_7$ and $h\geq 2$ admissible positions. Since $h\geq 2$, there must exist an index $i$, with $1\leq i\leq n-1$, such that every element in $\sigma_1\cdots \sigma_i$ is larger than every element to the right of $\pi_i$. We choose the smallest such $i$.

We observe that no occurrence of $p_7$ can begin to the left of $\sigma_{i+1}$ and end to the right of $\pi_i$. Since $i$ is minimal, there are $s_{i,a,1}(p_7)$ ways to choose $\sigma_1\cdots\sigma_i$, where $0\leq a\leq k$, and independently $s_{n-i,k-a,h-1}(p_7)$ ways to choose the remaining part of $\sigma$. Summing over all possible values of $i$ and $a$ yields (\ref{thm:class75-rec-3}). This completes our proof. 
\end{proof}

% In this subsection, we demonstrate that the recursive structure established in Theorem 1.6 naturally translates into a compact non-linear differential equation. This analytical perspective compresses the distributional behavior of Class 75 into a single formula.

\begin{coro}\label{cor:riccati}
  Let $F(t,q) := \sum_{n\ge 0, k\ge 0} s_{n,k}(p) q^k t^n$ be the bivariate generating function for the distribution of any chosen pattern $p$ in Class $75$, then it satisfies the following Riccati differential equation:
  \begin{equation}\label{eq:riccati}
    t^2 \frac{\partial F(t,q)}{\partial t} = -1+(1-2t+qt)F(t,q) + t(1-q)F(t,q)^2.
  \end{equation}
\end{coro}

\begin{proof}
  To begin with, for $h \ge 1$ we denote the generating function for the permutations with exactly $h$ admissible positions as $G_h(t,q) := \sum_{n \ge 1, k \ge 0} s_{n,k,h}(p) q^k t^n$. Noting that any permutation with $h$ admissible positions can be uniquely decomposed into $h$ irreducible (i.e., with only one admissible position) blocks, we deduce that
  $$G_h(t,q) = [G_1(t,q)]^h.$$
 On the one hand, multiplying by $q^k t^n$ and summing over all $n \ge 1$ and $k \ge 0$, we obtain
\begin{align*}
    F(t,q) - s_{0,0}(p)
    = \sum_{h \ge 1} G_h(t,q)
    = \sum_{h \ge 1} [G_1(t,q)]^h
    = \frac{G_1(t,q)}{1 - G_1(t,q)}.
\end{align*}
  Since $s_{0,0}(p) = 1$, we can solve for $G_1(t,q)$ to obtain $G_1(t,q) = 1 - \frac{1}{F(t,q)}$.

  On the other hand, we translate the main recurrence \eqref{thm:class75-rec-1} into a differential equation. Multiplying \eqref{thm:class75-rec-1} by $q^k t^n$ and summing over $n \ge 1$ and $k \ge 0$, we see that the left-hand side becomes $F(t,q) - 1$, while the right-hand side contains three terms that we analyze one by one. For the first term on the right-hand side, we apply the differential operator $t \frac{\partial}{\partial t}$ to see that
  \begin{equation}\label{eq:term 1}
    \sum_{n \ge 1, k \ge 0} n s_{n-1,k}(p) q^k t^n = t \frac{\partial}{\partial t} \big(t F(t,q)\big) = t F(t,q) + t^2 \frac{\partial F(t,q)}{\partial t}.
  \end{equation}
   In order to handle the second and third terms in \eqref{thm:class75-rec-1} that involve the multiplier $h$, we introduce a trivariate generating function $H(t,q,z) := \sum_{h \ge 1} z^h G_h(t,q)=\frac{z G_1(t,q)}{1 - z G_1(t,q)}$. Differentiating $H$ with respect to $z$ and evaluating at $z=1$ gives 
  \begin{align}\label{eq:H_derivative}
    \sum_{n \ge 1,\, k \ge 0,\, h \ge 1} h s_{n,k,h}(p) q^k t^n &= \left. \frac{\partial H(t,q,z)}{\partial z} \right|_{z=1} = \frac{G_1(t,q)}{\big(1 - G_1(t,q)\big)^2} = F(t,q)\big(F(t,q) - 1\big).
  \end{align}

  Now for the second term on the right-hand side, shifting the index $m = n-1$ and applying \eqref{eq:H_derivative} yields (note that $s_{0,k,h}(p)=0$ when $h\ge 1$):
  \begin{equation}\label{eq:term 2}
    - t \sum_{m \ge 0,\,k \ge 0,\,h \ge 1} h s_{m,k,h}(p) q^k t^m = -t F(t,q)\big(F(t,q) - 1\big).
  \end{equation}
  For the third term, a similar index shift on $k$ gives an extra factor of $q$:
  \begin{equation}\label{eq:term 3}
    qt \sum_{m \ge 0,\, j \ge 0,\, h \ge 1} h s_{m,j,h}(p) q^j t^m = qt F(t,q)\big(F(t,q) - 1\big).
  \end{equation}

  %  By multiplying both sides of \eqref{eq:riccati} by $t^k x^n$ and summing over all $n \ge 1$ and $k \ge 0$, we observe that the right-hand side is exactly the discrete Cauchy product (convolution) of the sequences corresponding to $h=1$ and $h-1$. Thus, we obtain:
  % \begin{equation}\label{eq:Gh}
  %   G_h(x,t) = G_1(x,t) \cdot G_{h-1}(x,t).
  % \end{equation}
  % By induction on $h$, this immediately yields , which structurally signifies 
  
  % Next, we rewrite (17) as $s_{n,k}(p) = \sum_{h=1}^n s_{n,k,h}(p)$ for $n \ge 1$. 

  Combining \eqref{eq:term 1}, \eqref{eq:term 2}, and \eqref{eq:term 3}, we arrive at:
  \begin{equation*}
    F(t,q) - 1 = t^2 \frac{\partial F(t,q)}{\partial t} + t F(t,q) - t(1-q)F(t,q)\big(F(t,q) - 1\big).
  \end{equation*}
  Rearranging the terms completes the proof of \eqref{eq:riccati}.
\end{proof}

\section{Concluding remarks}\label{sec:conclud}

The present work brings the classification of mesh patterns of length $2$ significantly closer to completion, but leaves a single outstanding equidistribution problem unresolved.\footnote{This problem was recently settled by Zhang and Zhao~\cite{ZZ2026}.} As explained above, resolving this final case would simultaneously settle both the remaining equidistribution and Wilf-equivalence classifications. Recall that the remaining open equidistribution (Class~69 in Remark~\ref{class-69-remark-equivalence}) concerns the $4$ patterns in the set $S_1:=\{\hspace{-1mm}\pattern{scale=0.5}{2}{1/1,2/2}{1/2,1/1,2/1,0/0},
\pattern{scale=0.5}{2}{1/1,2/2}{2/2,0/1,1/1,1/0},
\pattern{scale=0.5}{2}{1/1,2/2}{0/2,1/1,2/1,1/0},
\pattern{scale=0.5}{2}{1/1,2/2}{1/2,0/1,1/1,2/0}\}$. Reversing each of them we get another set of $4$ patterns $S_2:=\{\hspace{-1mm}\pattern{scale=0.5}{2}{1/2,2/1}{1/2,1/1,0/1,2/0},
\pattern{scale=0.5}{2}{1/2,2/1}{0/2,2/1,1/1,1/0},
\pattern{scale=0.5}{2}{1/2,2/1}{2/2,1/1,0/1,1/0},
\pattern{scale=0.5}{2}{1/2,2/1}{1/2,2/1,1/1,0/0}\}$, which is trivially equidistributed with anyone from $S_1$ over the entire set of permutations. Supported by computations upto $n=10$, we conjecture that some of the equidistributions among these $8$ patterns in $S_1\cup S_2$ continue to hold when restricted to the class of involutions. 

% \footnote{Note that the set of involutions, i.e., permutations whose cycles have length at most $2$, is not closed under reversal, and indeed $\pattern{scale=0.5}{2}{1/1,2/2}{1/2,1/1,2/1,0/0}$ is not equidistributed with $\pattern{scale=0.5}{2}{1/1,2/2}{0/2,1/1,2/1,1/0}$ over the set of involutions.} 

\begin{conj}
The patterns in the set $\{\hspace{-1mm}\pattern{scale=0.5}{2}{1/1,2/2}{1/2,1/1,2/1,0/0},
\pattern{scale=0.5}{2}{1/1,2/2}{2/2,0/1,1/1,1/0},
\pattern{scale=0.5}{2}{1/2,2/1}{2/2,1/1,0/1,1/0},
\pattern{scale=0.5}{2}{1/2,2/1}{1/2,2/1,1/1,0/0}\}$ are equidistributed on involutions. (The first two patterns, as well as the last two patterns, are trivially equidistributed via the composition of reverse and complement.)
\end{conj}

\begin{remark}
In the previous version of our manuscript~\cite{FFKLSS2026}, the conjecture above was stated for the four patterns in the set 
$S_1:=\{\hspace{-1mm}\pattern{scale=0.5}{2}{1/1,2/2}{1/2,1/1,2/1,0/0},
\pattern{scale=0.5}{2}{1/1,2/2}{2/2,0/1,1/1,1/0},
\pattern{scale=0.5}{2}{1/1,2/2}{0/2,1/1,2/1,1/0},
\pattern{scale=0.5}{2}{1/1,2/2}{1/2,0/1,1/1,2/0}\}$, 
overlooking the fact that the set of involutions (permutations whose cycles have length at most $2$) is not closed under reversal. 
Indeed, $\pattern{scale=0.5}{2}{1/2,2/1}{1/2,1/1,2/1,0/0}$ and $\pattern{scale=0.5}{2}{1/1,2/2}{1/2,0/1,1/1,2/0}$ are not equidistributed over the set of involutions. 
Consequently, Conjecture 1 in \cite{FFKLSS2026} is false. 
Although we are correcting this statement only now, the error was independently brought to our attention by Jiahao Zhang\footnote{Personal communication, August 16, 2026.} and by Wenlin Zhang and Haobo\footnote{Personal communication, October 11, 2026.}, whom we gratefully acknowledge.
\end{remark}

At present, it is not clear whether the original conjectured equidistribution or its restriction to involutions is the more structurally meaningful statement. It is conceivable that the involution-restricted version admits a more direct or simpler proof, or alternatively, that the full equidistribution is the more natural and tractable object. Clarifying this relationship remains an interesting direction for future research.

% ---------------------------------------------------------
\section*{\bf Acknowledgments}
 Shishuo Fu was partially supported by the Fundamental Research Funds for the Central Universities (grant no.~2025CDJ-IAISYB-008).

\end{document}